\documentclass[11pt,reqno]{article}   
\usepackage{fullpage}

\usepackage{microtype}
\usepackage[unicode=true]{hyperref}
\usepackage[british]{babel}
\usepackage{amsmath}
\usepackage{cite}
\usepackage{amsfonts}
\usepackage{amssymb}
\usepackage{amsthm}
\usepackage{cleveref}  
\usepackage{authblk}
\usepackage{mypack} 
\usepackage[pdftex]{color,graphicx}
\usepackage{adjustbox}
\usepackage{soul}
\usepackage{comment}
\usepackage{bbold}
\usepackage{tikz} 
\usetikzlibrary{decorations.pathreplacing,angles,quotes}
\usetikzlibrary{matrix,arrows,decorations.pathmorphing}
\usepackage{tikz-cd}
\usetikzlibrary{arrows}
\usetikzlibrary{decorations.markings}
\usepackage{braket} 

\usepackage{bbold}
\usepackage{diagbox}
\usepackage{caption}
\usepackage{subcaption}
\usepackage{pdfpages} 
\usepackage{blkarray}
\usepackage{centernot}
\usepackage{mathtools}
\usepackage{stmaryrd}
\usepackage{soul}  
\usepackage{enumitem}

\definecolor{bluegray}{rgb}{0.4, 0.6, 0.8}
\definecolor{turquoise}{rgb}{0.2, 0.7, 0.6}

\usepackage{todonotes}
\newcounter{commcount}

\newcommand\cpll[1]{{\color{orange}#1}}	

\usepackage{todonotes}
\usepackage{xcolor}

\title{Homology and carrier faces of simplicial distributions}
 
\author{Aziz Kharoof\footnote{aziz.kharoof@bilkent.edu.tr}}
\author{Chung-Ping Lai\footnote{cp.lai@bilkent.edu.tr}} 
\affil{{\small{Department of Mathematics, Bilkent University, Ankara, Turkey}}}

\date{\today}

\begin{document}

\maketitle
\begin{abstract}
We study the facial structure of simplicial distribution polytopes using topological and algebraic methods. In the theory of simplicial distributions, measurement and event spaces are represented by simplicial sets, and compatible probability assignments form a convex polytope. Given a simplicial distribution, it is known that its supports determine its carrier face. Moreover, the supports form a simplicial subset of the event space. This naturally raises the question of how the topology of this simplicial subset is related to the corresponding carrier face. Our main objective is to relate the dimension of the carrier face to the topology and combinatorial structure of the space of supports. For $n$-facet-paired measurement spaces, we first express the carrier dimension in terms of the nullity of an associated facet matrix. When the measurement space is additionally orientable, we show that the carrier dimension is determined by the top-dimensional homology of the space of supports, yielding, in particular, a homological criterion for detecting vertices. 
We then develop an algebraic reduction based on a strong
connectivity relation on the generating simplices of the space
of supports. This construction yields a reduced facet matrix,
indexed by the strongly connected components, whose null space
is isomorphic to that of the original facet matrix.
As a consequence, both the carrier dimension and extremality can be determined from the topology of the support and the way its strongly connected components are glued together.
\end{abstract}

\tableofcontents

\section{Introduction}

Probability distributions arising from measurement scenarios often form convex sets with rich combinatorial and geometric structure. In the finite case, distributions satisfying the relevant compatibility conditions form a polytope, and understanding its vertices and faces is a natural problem \cite{pironio2011extremal}. Such probability models arise, in particular, in the study of non-locality and contextuality in quantum foundations  \cite{bell64,KS67,brunner2014bell,budroni2022kochen}.

The framework of simplicial distributions, introduced in \cite{okay2022simplicial}, provides a topological approach to these probability models. In this framework, measurement scenarios and outcome spaces are represented by simplicial sets, allowing tools from simplicial and algebraic topology to be applied to the study of probability distributions. Simplicial sets are combinatorial models of spaces and form one of the basic objects of modern homotopy theory \cite{goerss2009simplicial}. This perspective has led to several developments in which topological and algebraic methods are used to study structural properties of simplicial distributions; see, for example, \cite{kharoof2023topological,kharoof2023homotopical,kharoof2026geometry,kharoof2026vertex}.

In this paper, we focus on the geometry of the polytope of simplicial distributions and investigate how its facial structure can be described using topological and combinatorial information. A complementary description of the vertices and faces of these polytopes is obtained by recording only which events have nonzero probability. Although this support information forgets the actual probability values, it determines the face structure of a polytope in standard form \cite{abramsky2016possibilities}. We study how the homology of the support of a simplicial distribution can be used to compute the dimension of the face determined by that support.

We work in the framework of simplicial distributions on a simplicial map
$$
f\colon E\to X,
$$
which describes a space $X$ of measurement contexts together with a space $E$
of events, where the fiber over a simplex $\sigma$ of $X$ consists of
the outcomes of the corresponding context. A simplicial distribution
on $f$ assigns a probability distribution $p_\sigma$ to each fiber,
compatibly with the face and degeneracy maps. In particular, the face
maps express compatibility under marginalization. We denote the set of
such distributions by $\sDist(f)$. When $X$ is finitely generated and
$f$ has finite fibers, this set is a polytope in standard form.

For $p\in\sDist(f)$, the \emph{carrier face} $\carr(p)$ is the smallest face containing $p$, or equivalently, the unique face whose relative interior contains $p$. In particular, $p$ is a vertex of $\sDist(f)$ if and only if $\dim(\carr(p))=0$.
More generally, the dimension of $\carr(p)$ measures the number of independent directions along which $p$ can vary while remaining in the same face. Thus, the carrier dimension controls both extremality and the freedom in convex decompositions.

The support data naturally assemble into a simplicial set. Given
$p\in\sDist(f)$, its \emph{space of supports} is the simplicial subset
$E(p)\subseteq E$ defined by
$$
E(p)_n
=
\set{e\in E_n:\; p_{f_n(e)}(e)\neq 0}.
$$
The description of carrier faces in terms of supports gives
$$
\carr(p)
=
\set{q\in\sDist(f):\; E(q)\subseteq E(p)}.
$$
Consequently, the carrier face can be studied by restricting $f$ to the space of supports, obtaining 
$$
f|_{E(p)}\colon E(p)\to X.
$$
In fact, it is proved in \cite{kharoof2026possibilistic} that
$f|_{E(p)}$ is a simplicial bundle scenario in the sense of
\cite{barbosa2023bundle}, and therefore possesses useful geometric
properties. Our aim is to determine the dimension of $\carr(p)$ from the
topology of $E(p)$ and from the way its simplices are glued together.

Our main setting is that of $n$-facet-paired measurement spaces $X$. Roughly speaking, these are  finite $n$-dimensional simplicial sets obtained by gluing $n$-simplices along their $(n-1)$-dimensional faces so that every such face is shared by exactly two $n$-simplices, and the $n$-simplices are connected through these shared faces. {Such simplicial sets resemble triangulated closed manifolds.}  This condition provides
the type of regular gluing structure needed to relate the geometry of the
polytope of simplicial distributions to the topology of their spaces of
supports.
Our main examples are:
\begin{itemize}
\item the boundary of a simplex of dimension at least two. This scenario plays a key role in connection with the classical theorem of Vorob'ev \cite{vorob1962consistent}, which characterizes when compatible families of probability distributions admit a global extension;
\item the usual measurement spaces of Bell scenarios \cite{brunner2014bell} with two measurements per observer and at least two observers.
\end{itemize}

In this setting, the topology of the space of supports completely determines the
dimension of the carrier face. More precisely, we obtain the following
homological characterization.

\begin{thm}[Homological characterization]\label{thm:intro-homology}
Let $n\geq 1$, let $X$ be an orientable $n$-facet-paired simplicial set, and let $f\colon E\to X$ be a simplicial map. Then, for every $p\in\sDist(f)$,
$$
\dim(\carr(p))
=
\rank H_n(E(p))-1,
$$
where homology is taken with integer coefficients.
\end{thm}

The theorem therefore provides a way to determine carrier dimensions,
and in particular to test extremality, without first enumerating the
vertices of the polytope. It identifies the remaining probabilistic
freedom within the carrier face with the top-dimensional homology of the
space of supports. In particular, under the hypotheses of the theorem,
$$
p \text{ is a vertex } \iff  
\rank H_n(E(p))=1.
$$

We also give an algebraic description that does not require
orientability. For a direction $\chi$ on the graph of generating
$n$-simplices of $X$ joined along $(n-1)$-faces, we define a facet matrix $A(f|_{E(p)},\chi)$
(Definition \ref{def:Afchi}) that
records the marginal compatibility equations on the support.
Some of these equations simply force two individual probabilities to be equal.
The strong connectivity relation, defined first in \cite{kharoof2026possibilistic}, groups the generating simplices of
$E(p)$ along precisely these elementary equalities.
Replacing coordinates corresponding to simplices in the same
strongly connected component by a single coordinate yields the
\emph{reduced facet matrix} $C(f|_{E(p)},\chi)$, which records how
distinct strongly connected components are glued together along
their $(n-1)$-faces.
\begin{thm}[Algebraic characterization]\label{thm:intro-algebraic}
Let $n\geq 1$, let $X$ be an $n$-facet-paired simplicial set,
and let $f\colon E\to X$ be a simplicial map.
For every $p\in\sDist(f)$ and every direction $\chi$ on $X$, we have
$$
\dim(\carr(p))
=\dim\bigl(\Null(A(f|_{E(p)},\chi))\bigr)-1
=\dim\bigl(\Null(C(f|_{E(p)},\chi))\bigr)-1.
$$
\end{thm}
If the support has $k$ strongly connected components, this yields
$$
p\text{ is a vertex}
\quad\Longleftrightarrow\quad
\rank \left(C(f|_{E(p)},\chi)\right)=k-1.
$$
Since each column of the reduced facet matrix can correspond to
multiple $n$-simplices and many $(n-1)$-simplices do not appear
as rows, $C(f|_{E(p)},\chi)$ can be substantially smaller than
the original facet matrix $A(f|_{E(p)},\chi)$, making it
particularly useful for computations.

The paper is organized as follows. Section~\ref{sec:polydist} recalls the necessary background on polytopes and simplicial sets and introduces simplicial distributions. Section~\ref{sec:spacesupp} studies the spaces of supports of simplicial distributions, introduces $n$-facet-paired simplicial sets and $n$-facet-paired bundle scenarios, and expresses the dimension of the carrier of a simplicial distribution in terms of the nullity of an associated facet matrix. Section~\ref{sec:Homolgcar} recalls the homology of simplicial sets and relates the carrier dimension to the top-dimensional homology of the space of supports. 
Finally, Section~\ref{sec:strongconnct} introduces the strong
connectivity relation and constructs the reduced facet matrix.

\paragraph{Acknowledgments.}
This work is supported by the Air Force Office of Scientific Research (AFOSR) under award number  FA9550-24-1-0257.


\section{The polytope of simplicial distributions} \label{sec:polydist}
In this section, we introduce simplicial distributions and study the polytope they form. We begin by recalling the basic notions concerning polytopes that will be used throughout the paper. We then review the relevant simplicial-set and distribution-monad constructions, introduce simplicial distributions, and develop a language for describing them in low dimensions.

\subsection{Polytopes in standard form}
We first recall some basic notions concerning polytopes. In particular, we characterize the faces of polytopes in standard form in terms of the support of their points.

\begin{defn}\label{def:polystan}
A \emph{polytope} in the Euclidean space $\RR^n$ is a bounded set of the form
$$
L=\{x\in\RR^n:\; Ax\leq b\},
$$
where $A$ is a real $m\times n$ matrix and $b\in\RR^m$. 

A polytope $L$ is called a \emph{polytope in standard form} if
$$
L=\{x\in\RR^n:\; Ax=b,\ x_i\geq 0\text{ for all }i\},
$$
where $A$ is a real $m\times n$ matrix, $b\in\RR^m$.
\end{defn}

\begin{defn}\label{def:face}
Let $L\subseteq\RR^n$ be a polytope. A \emph{face} of $L$ is a subset
$$
F=L\cap H,
$$
where $H=\{x\in\RR^n:\; c^\top x=\alpha\}$
is a hyperplane such that $c^\top x\leq\alpha$ for all $x\in L$.
\end{defn}
\begin{defn}\label{def:vertexxx}
An element $v$ of a polytope $L$ is called a \emph{vertex} if whenever $v = tb + (1-t)c$ for some $b, c \in L$ and $0 < t < 1$, then $b = c = v$. 
\end{defn}
\begin{defn}\label{def:affinindep}
A finite set of points $\{v_1, \dots, v_k\} \subset \mathbb{R}^n$ is said to be \emph{affinely independent} 
if the only real coefficients $\lambda_1, \dots, \lambda_k$ satisfying
    \[
    \sum_{i=1}^k \lambda_i v_i = 0
    \quad \text{and} \quad
    \sum_{i=1}^k \lambda_i = 0
    \]
    are $\lambda_1 = \cdots = \lambda_k = 0$. 
\end{defn}
\begin{defn}\label{def:dimF}
Let $L$ be a polytope and let $F$ be a face of $L$. If
$A\subseteq F$ is a maximal affinely independent subset of $F$, then the
\emph{dimension} of $F$ is defined by
$$
\dim(F) := |A| - 1.
$$
\end{defn}
The dimension above is well defined, since any two maximal affinely independent subsets of $F$
have the same cardinality.

\begin{ex}\label{ex:dimvertex=0}
If $F=\set{v}$, then $v$ is a vertex and $\dim(F)=0$.
\end{ex}

\begin{defn}\label{def:convaffin}
For a subset $X \subseteq \RR^n$,
\begin{itemize}
\item We write $\Conv(X)$ for the \emph{convex hull of $X$}:
$$
\Conv(X) = \left\{\sum_{i=1}^k \alpha_i x_i :\; k \in \NN, x_i \in X, \alpha_i \geq 0, \sum_{i=1}^k \alpha_i = 1\right\}.
$$
\item We write $\Aff(X)$ for the \emph{affine hull} of $X$:
$$
\Aff(X) = \left\{\sum_{i=1}^k \alpha_i x_i :\; k \in \NN, x_i \in X, \alpha_i \in \RR, \sum_{i=1}^k \alpha_i = 1\right\}.
$$
\end{itemize} 
\end{defn}

\begin{defn}\label{def:carrier}
Let $L$ be a polytope and let $F$ be a face of $L$. The
\emph{relative interior} of $F$, denoted by $\relint(F)$, consists of those
points $x\in F$ for which there exists $\varepsilon>0$ such that $y\in F$
whenever $y\in\Aff(F)$ and $d(x,y)<\varepsilon$.

For a point $x\in L$, the \emph{carrier face} of $x$, denoted by
$\carr(x)$, is the unique face $F$ of $L$ such that $x\in\relint(F)$.
\end{defn}

\begin{defn}\label{def:GenRelation}
Let $L\subseteq\RR^n$ be a polytope in standard form. For $x,y\in L$, we write $y\preceq x$
if, for every $1\leq i\leq n$,
$$
y_i\neq 0 \quad\Longrightarrow\quad x_i\neq 0.
$$
\end{defn}
This relation $\preceq$ is a preorder on the polytope $L$.

\begin{pro}\label{pro:yleqxVsuppcarr}
(\!\!\cite[Proposition 5.1]{abramsky2016possibilities}) 
Given a polytope in standard form $L$ and $x \in L$, we have 
$$
\carr(x)=\set{y\in L :\;  {y} \preceq {x} }.
$$
\end{pro}

\subsection{Simplicial distributions}

{We begin by recalling the basic notions concerning simplicial sets and the
distribution monad that will be used throughout this section.}

Let $\Delta$ denote the \emph{simplex category}, whose objects are the finite ordered sets
$$
[n]=\{0,1,\dots,n\},
\qquad n\geq 0,
$$
and whose morphisms are order-preserving maps. A \emph{simplicial set} is a contravariant functor
$$
X\colon \Delta^{\op}\to \catSet,
$$
where $\catSet$ denotes the category of sets. Equivalently, a simplicial set $X$ consists of a family of sets $(X_n)_{n\ge 0}$ together with face maps 
$d_i\colon X_n \to X_{n-1}$ for $0\le i \le n$ and degeneracy maps 
$s_i \colon X_n \to X_{n+1}$ for $0\le i \le n$ that satisfy the standard simplicial identities (see \cite{friedman2008elementary}).

An $n$-simplex $x\in X_n$ is called \emph{degenerate} if it lies in the
image of a degeneracy map; otherwise, it is called \emph{nondegenerate}.
We denote the set of nondegenerate simplices in $X_n$ by $X_n^\circ$.
A nondegenerate simplex is called a \emph{generating simplex} if it is not a
face of any other nondegenerate simplex.

An important example is the \emph{standard $n$-simplex} $\Delta^n$, defined as the simplicial set represented by $[n]$, i.e.
$$
\Delta^n(-)=\Hom_{\Delta}(-,[n]).
$$
The $k$-simplices of $\Delta^n$ are
$$
(\Delta^n)_k=
\left\{
[i_0,\ldots,i_k]
:\;
0\le i_0\le\cdots\le i_k\le n
\right\},
$$
where the face and degeneracy maps are given by deleting and repeating entries, respectively.
The simplex $[0,1,\ldots,n]$ is the unique generating simplex of $\Delta^n$.

A \emph{simplicial map} $f \colon X \to Y$ may be viewed either as a natural transformation between the corresponding functors or, equivalently, as a family of functions 
$f_n \colon X_n \to Y_n$ commuting with all face and degeneracy maps. For a simplex $\sigma \in X_n$ we usually write $f_{\sigma}$ instead of 
$f_n(\sigma)$ for {its image in $Y_n$.} We denote the category of simplicial sets by $\catsSet$.

A simplicial set $X$ is called \emph{$n$-dimensional} if for every $k>n$, every simplex in $X_k$ is degenerate.
For example, a $1$-dimensional simplicial set consists only of vertices and edges
(up to degeneracies), and can be viewed as a directed multigraph.

We also use the distribution monad $D$ (see \cite{jacobs2011probabilities}). For a set $X$, the set $D(X)$ consists of all finitely supported probability distributions on $X$.
$$
D(X)=\set{P\colon X \to \RR_{\geq 0}: \; \sum_{x\in X} P(x)=1 \text{ and } |\set{x:\; P(x)\neq 0}|< \infty}
$$
A function $g \colon X_1 \to X_2$ induces a pushforward map 
$D(g)\colon  D(X_1) \to D(X_2)$ given by marginalization. That means, for every $P \in D(X_1)$ and $x\in X_2$, we have:
$$
D(g)(P)(x)=\sum_{y:\;g(y)=x} P(y).
$$
This construction extends level-wise to simplicial sets, yielding a monad 
$D\colon \catsSet \to \catsSet$ (see \cite[Proposition 2.6]{kharoof2022simplicial}). So we have a unit 
$\delta_X\colon  X \to D(X)$ for every simplicial set $X$, assigning to each simplex the corresponding Dirac distribution.

\begin{defn}
Let $f\colon E\to X$ be a simplicial map. 
A \emph{simplicial distribution} on $f$ is a simplicial map $p \colon X \to D(E)$ making the following diagram commute:
$$
\begin{tikzcd}[column sep=huge,row sep=large]
 & D(E) \arrow[d,"D(f)"] \\
X\arrow[ru,"p"] \arrow[r,"\delta_X"'] & D(X).
\end{tikzcd}
$$
Equivalently, a simplicial distribution is a family of distributions
$$
\{p_{\sigma} \in D(f^{-1}(\sigma)):\; \sigma\in X_n\}_{n\geq 0}
$$
satisfying the compatibility conditions induced by the simplicial structure of $X$. 

We denote the set of simplicial distributions on $f$ by $\sDist(f)$. 
The simplicial set $X$ is usually called the \emph{measurement space}.
\end{defn}
\begin{pro}\label{pro:sdistfinstanform}
Let $f\colon E \to X$ be a simplicial map, where $X$ is finitely generated and, for every simplex $x\in X_n$, the fiber $f_n^{-1}(x)$ is finite. Then the set of simplicial distributions $\sDist(f)$ is a polytope in standard form.
\end{pro}
\begin{proof}
The proof is similar to that of \cite[Proposition 3.6]{kharoof2026vertex}.
\end{proof}

Now we {provide} some examples of measurement spaces, scenarios, and simplicial distributions. 

\begin{defn}\label{def:BXm}
Let $X$ be a simplicial set and let
$
\vec{m}=(m_x)_{x\in X_0}
$
be a tuple of natural numbers with $m_x\geq 2$. We define a simplicial set
$E(X,\vec m)$ as follows.

For $\sigma\in X_j$, denote its vertices by
$$
v_i(\sigma)
=
d_0^{\,i}d_{i+1}^{\,j-i}(\sigma)\in X_0,
\qquad
0\leq i\leq j.
$$
Then
$$
E(X,\vec m)_j 
=
\bigsqcup_{\sigma\in X_j}
\{\sigma\}\times
\ZZ_{m_{v_0(\sigma)}}\times\cdots\times
\ZZ_{m_{v_j(\sigma)}}.
$$

The structure maps are given by
$$
\begin{aligned}
d_i(\sigma,(a_0,\dots,a_j))
&=
(d_i(\sigma),(a_0,\dots,\widehat{a_i},\dots,a_j)),\\
s_i(\sigma,(a_0,\dots,a_j))
&=
(s_i(\sigma),(a_0,\dots,a_i,a_i,\dots,a_j)).
\end{aligned}
$$

The evident projection map induces a simplicial map
$$
f_{X,\vec{m}}\colon E(X,\vec{m})\to X.
$$
In the case that $m_x=m$ for every $x\in X_0$, we write $f_{X,m}\colon E(X,m)\to X$.
\end{defn}

\begin{ex}\label{ex:simpdisanydim}
For a $2$-dimensional simplicial set $X$, a simplicial distribution on $f_{X,\vec{m}}$ is described as follows:
\begin{itemize}
    \item To each $0$-simplex $x \in X_0$, we assign a distribution $p_x$ on $\ZZ_{m_x}$, which can be written as a vector
    $$
    \begin{pmatrix}
    p^0_x \\
    \vdots\\
    p^{m_x-1}_x
    \end{pmatrix}
    $$
    \item To each $1$-simplex $\sigma \in X_1$, we assign a distribution $p_\sigma$ on  $\ZZ_{m_{d_1(\sigma)}}\times
\ZZ_{m_{d_0(\sigma)}}$, which can be written as a matrix 
    $$
     \begin{pmatrix}
    p_{\sigma}^{0,0} \;\dots \dots \; p_{\sigma}^{0,m_{d_0(\sigma)}-1}\\
    \vdots \;\;\;\;\vdots \;\;\;\; \vdots \\
    p_{\sigma}^{m_{d_1(\sigma)}-1,0} \; \dots \dots \; p_{\sigma}^{m_{d_1(\sigma)}-1,m_{d_0(\sigma)}-1}
    \end{pmatrix}
    $$
Restricting to the endpoints of $\sigma$ (equivalently, taking marginals) yields
$$
    p_{d_1(\sigma)}=\begin{pmatrix}
    p_{\sigma}^{0,0} +\dots+ p_{\sigma}^{0,m_{d_0(\sigma)}-1}\\
    \vdots  \\
    p_{\sigma}^{m_{d_1(\sigma)}-1,0} + \dots + p_{\sigma}^{m_{d_1(\sigma)}-1,m_{d_0(\sigma)}-1}
    \end{pmatrix}\;,\;
    p_{d_0(\sigma)}= \begin{pmatrix}
    p_{\sigma}^{0,0} +\dots + p_{\sigma}^{m_{d_1(\sigma)}-1,0}\\
    \vdots \\
    p_{\sigma}^{0,m_{d_0(\sigma)}-1} + \dots + p_{\sigma}^{m_{d_1(\sigma)}-1,m_{d_0(\sigma)}-1}
    \end{pmatrix}
    $$
\end{itemize}
For each $2$-simplex $\sigma\in X_2$ we have a
distribution
$$
p_\sigma
\text{ on }
\ZZ_{m_{v_0(\sigma)}}\times
\ZZ_{m_{v_1(\sigma)}}\times
\ZZ_{m_{v_2(\sigma)}} .
$$
We may represent this distribution as $m_{v_2(\sigma)}$ matrices indexed by
$k=0,\dots,m_{v_2(\sigma)}-1$, where the $k$-th matrix is
$$
\begin{pmatrix}
p_\sigma^{0,0,k} & \dots & p_\sigma^{0,m_{v_1(\sigma)}-1,k}\\
\vdots & & \vdots\\
p_\sigma^{m_{v_0(\sigma)}-1,0,k} & \dots &
p_\sigma^{m_{v_0(\sigma)}-1,m_{v_1(\sigma)}-1,k}
\end{pmatrix}.
$$
The face maps correspond to taking marginals:
\begin{itemize}
\item $p_{d_2(\sigma)}^{i,j}=\sum_k p_\sigma^{i,j,k}$,
\item $p_{d_1(\sigma)}^{i,k}=\sum_j p_\sigma^{i,j,k}$,
\item $p_{d_0(\sigma)}^{j,k}=\sum_i p_\sigma^{i,j,k}$.
\end{itemize}
\end{ex}

\begin{defn}\label{def:sB(k,vecr)}
Let $k$ be a natural number and let $r_1,\dots,r_k$ be natural numbers. 
Let $A_1,\dots,A_k$ be pairwise disjoint sets such that 
$|A_i|=r_i$ for every $1 \leq i \leq k$. 
We define $sB(k,\vec r)$ to be the simplicial set whose set of $n$-simplices is
$$
sB(k,\vec{r})_n=\bigsqcup_{1\leq i_0\leq \dots \leq i_{n} \leq k} A_{i_0}\times \dots \times A_{i_{n}}.
$$
The face and degeneracy maps are given by deleting and repeating coordinates,
respectively. More precisely, 
$$
d_i(x_0,\dots,x_n) = (x_0,\dots,\widehat{x_i},\dots,x_n),
$$
and
$$
s_i(x_0,\dots,x_n) = (x_0,\dots,x_i,x_i,\dots,x_n).
$$
If $r_i=r$ for every $1\leq i \leq k$, then we will write $sB(k,r)$.
\end{defn}
{As a measurement space, $sB(k,\vec r)$}
represents a Bell
scenario with $k$ observers, where observer $i$ has $r_i$ measurements.

Note that $sB(k,\vec r)$ is a $(k-1)$-dimensional simplicial set, generated
by the simplices in
$$
A_1\times A_2\times\cdots\times A_k.
$$

\begin{ex}\label{ex:CHSHscenario}
Geometrically, the measurement space $sB(2,2)$ is a square. The simplicial map $f_{sB(2,2),2}$ represents 
the well-known CHSH scenario \cite{chsh69}.
\end{ex}
%


%

%

\begin{ex}\label{ex:233bellvertex}
The $(2,3,3)$ Bell scenario, i.e., two observers with three measurements per
observer and three outcomes per measurement, can be represented by the
simplicial map $f_{sB(2,3),3}$. 
The measurement space $sB(2,3)$ is a $1$-dimensional simplicial set whose generating simplices
are the elements of 
$$
A_1\times A_2,
$$
where 
$A_1=\{x_1,x_2,x_3\}$ and $A_2=\{y_1,y_2,y_3\}$.

Here is an example, introduced in \cite{jones2005interconversion} and discussed in \cite{kharoof2024extremal}, of a simplicial distribution $p$ on
$f_{sB(2,3),3}$ that is a vertex:

\noindent The distributions $p_{x_1,y_1}$, $p_{x_1,y_3}$, $p_{x_2,y_2}$, and $p_{x_3,y_2}$ are given by
    $$
    \begin{pmatrix}
    1/2 & 0 & 0 \\
    0 & 1/4 & 0\\
    0 & 0 & 1/4
    \end{pmatrix}.
    $$
The distribution $p_{x_1,y_2}$ is given by
    $$
    \begin{pmatrix}
    1/4 & 1/4 & 0 \\
    1/4 & 0 & 0 \\
    0 & 0 & 1/4
    \end{pmatrix}.
    $$
The distributions $p_{x_2,y_1}$ and $p_{x_3,y_1}$ are given by
    $$
    \begin{pmatrix}
    1/4 & 1/4 & 0 \\
    0 & 0 & 1/4\\
    1/4 & 0 & 0
    \end{pmatrix}.
    $$
The distribution $p_{x_2,y_3}$ is given by
$$
\begin{pmatrix}
0 & 1/4 & 1/4\\
1/4 & 0 & 0\\
1/4 & 0 & 0
\end{pmatrix}.
$$
The distribution $p_{x_3,y_3}$ is given by
$$
\begin{pmatrix}
1/4 & 0 & 1/4\\
0 & 1/4 & 0\\
1/4 & 0 & 0
\end{pmatrix}.
$$
\end{ex}

%
%
%
%

%

%
\section{The space of supports of a simplicial distribution}\label{sec:spacesupp}
In this section, we introduce the notion of an $n$-facet-paired simplicial set, which will serve as the measurement space in our main results. Two important examples of such measurement spaces are the Bell scenarios with two measurements per observer and the boundaries of simplices. We then associate a matrix to the space of supports of a simplicial distribution and derive a formula for the dimension of its carrier in terms of the nullity of this matrix.

\subsection{The geometrical properties of the space of supports}

We first define the space of supports and show that it forms a simplicial set. We then recall the notion of a bundle scenario and show that the restriction of a simplicial map to the space of supports is a bundle scenario. We also introduce the notions of an $n$-facet-paired simplicial set and an $n$-facet-paired bundle scenario. Finally, we establish some useful properties of the space of supports when the measurement space is $n$-facet-paired.

\begin{defn}\label{def:eventspaceofp}
Given a simplicial distribution $p$ on a simplicial map 
$f \colon E \to X$, we define the \emph{space of supports of $p$} to be the simplicial
subset $E(p) \subseteq E$ given in degree $n$ by
$$
E(p)_n=\set{e \in E_n:\; p_n(f_n(e))(e)\neq 0}
$$
\end{defn} 

The simplicial set $E(p)$ is well defined. To see this, let
$\theta\colon[k]\to[n]$ be a morphism in $\Delta$, let $\theta^\ast$
denote the corresponding simplicial structure map, and let
$e\in E(p)_n$. Since $f$ and $p$ are simplicial maps, we have
$$
\begin{aligned}
p_k\bigl(f_k(\theta^\ast(e))\bigr)(\theta^\ast(e))
&=
D(\theta^\ast)\bigl(p_n(f_n(e))\bigr)(\theta^\ast(e))\\
&=
\sum_{e':\;\theta^\ast(e')=\theta^\ast(e)}
p_n(f_n(e))(e').
\end{aligned}
$$
Since all terms in the sum are nonnegative and the term corresponding to
$e'=e$ is nonzero, we have
$$
p_k\bigl(f_k(\theta^\ast(e))\bigr)(\theta^\ast(e))
\geq
p_n(f_n(e))(e)>0.
$$
Hence $\theta^*(e)\in E(p)_k$ as desired.
\begin{defn}\label{def:RLPofBundle}
Given simplicial maps $f\colon E\to X$ and $g\colon Y\to Z$, we say that
$f$ has the \emph{right-lifting property with respect to $g$} if a 
diagonal map making the diagram commute exists in every commutative square of the form
$$
\begin{tikzcd}[column sep=huge,row sep=large]
Y \arrow[d,"g"'] \arrow[r] & E\arrow[d,"f"] \\
Z \arrow[ru,dashed] \arrow[r] & X.
\end{tikzcd}
$$

A simplicial map $f\colon E\to X$ is called a \emph{bundle scenario} if,
for every morphism $\theta\colon[k]\to[n]$ in $\Delta$, it has the
right-lifting property with respect to the induced map $\theta_\ast\colon\Delta^k\to\Delta^n$.

Equivalently, $f$ satisfies the following lifting conditions:
\begin{enumerate}
\item \emph{Surjective:} $f$ has the right-lifting property with respect to
all maps $\emptyset\to\Delta^n$.

\item \emph{Locally surjective:} $f$ has the right-lifting property with
respect to all coface maps $d^i\colon\Delta^{n-1}\hookrightarrow\Delta^n$.

\item \emph{Discrete over vertices:} $f$ has the right-lifting property
with respect to all codegeneracy maps $s^i\colon\Delta^n\to\Delta^{n-1}$.
\end{enumerate}

These properties were first introduced in
\cite[Definition 4.1]{barbosa2023bundle}.
\end{defn}
\begin{pro}\label{pro:f|Ebundle}
Given a simplicial distribution $p\in\sDist(f)$, the simplicial map
$$
f|_{E(p)}\colon E(p)\to X
$$
is a bundle scenario with finite fibers.
\end{pro}

\begin{proof}
See \cite[Proposition 3.7]{kharoof2026possibilistic}.
\end{proof}

{We now define the $n$-facet-paired measurement spaces used in our main results. Our definition shares some features with Voevodsky's notion of a simplicial pseudomanifold~\cite[\S 3.0]{Voevodsky1990Flags}, while imposing additional restrictions specific to our setting.
\begin{defn}\label{def:facet-paired}
Let $n\geq 1$. A simplicial set $X$ is called \emph{$n$-facet-paired} if 
$$
X \cong {\colim}_{i\in \mathcal I} F(i), 
$$
where $F\colon \mathcal I \to \catsSet$ is a functor satisfying the following conditions:

\begin{enumerate}
    \item The indexing category $\mathcal I$ is finite and connected. In addition, for every pair of objects $i,j\in \mathcal I$,
    $$
    \Hom_{\mathcal I}(i,i)=\{\Id_i\},
    \qquad
    |\Hom_{\mathcal I}(i,j)|\in\{0,1\}
    \quad \text{for } i\neq j.
    $$

    \item For every object $i\in \mathcal I$, either
    $F(i)=\Delta^n$ or $F(i)=\Delta^{n-1}$. Moreover, 
    the image under $F$ of every non-identity morphism $i\to j$ in $\mathcal I$ is a coface map
    $$
    d^k\colon \Delta^{n-1}\longrightarrow \Delta^n
    $$
    for some $0\leq k\leq n$.
    \item Every object $i\in\mathcal I$ with $F(i)=\Delta^{n-1}$
    is the source of exactly two distinct non-identity morphisms
    $$
    i\to j_1
    \qquad\text{and}\qquad
    i\to j_2,
    $$
    where $j_1\neq j_2$.
    Thus each copy of $\Delta^{n-1}$ is identified with a facet of exactly two distinct copies of $\Delta^n$.
    \item Every object $j\in\mathcal I$ with $F(j)=\Delta^n$ is the target of exactly $n+1$ distinct non-identity morphisms
    $$
    i_k\longrightarrow j,
    \qquad
    0\leq k\leq n,
    $$
    such that $F(i_k\to j)$ is the coface $d^k$.
    In particular, for each copy of $\Delta^n$, every facet inclusion is recorded exactly once in the diagram.
\end{enumerate}
\end{defn}

Our first main example of an $n$-facet-paired simplicial set is
$sB(k,2)$ (Definition~\ref{def:sB(k,vecr)}). We now give a second main
example.

\begin{ex}\label{ex:boundaryofsimplex}
The \emph{boundary of the standard $n$-simplex}, denoted by
$\partial\Delta^n$, is the simplicial subset of $\Delta^n$ consisting of
all simplices except the unique generating $n$-simplex
$[0,1,\ldots,n]$ and all of its degeneracies. Equivalently, it is the
simplicial subset generated by the $(n-1)$-faces
$$
[1,\ldots,n],\,
[0,2,\ldots,n],\,
\ldots,\,
[0,\ldots,n-1].
$$
In fact, $\partial\Delta^n$ is an $(n-1)$-facet-paired simplicial set.
\end{ex}

We record some immediate properties of $n$-facet-paired simplicial sets.
\begin{pro}\label{pro:facet-paired-properties}
Let $X$ be an $n$-facet-paired simplicial set. Then the following statements hold:
\begin{enumerate}
    \item $X$ is an $n$-dimensional simplicial set which is finitely generated by its
    nondegenerate $n$-simplices.

    \item For every nondegenerate simplex $\sigma\in X_n$ and every
    $0\leq i\leq n$, the face $d_i(\sigma)$ is nondegenerate.

    \item For every nondegenerate simplex $\sigma\in X_n$, $i\neq j$ implies $d_i(\sigma)\neq d_j(\sigma)$.

    \item Any two nondegenerate $n$-simplices can be joined by a chain of
$n$-simplices such that consecutive simplices in the chain are glued along
a nondegenerate $(n-1)$-face.
\end{enumerate}
\end{pro}

\begin{defn}
A bundle scenario $f\colon E\longrightarrow X$
is called \emph{$n$-facet-paired} if $f$ has finite fibers and $X$ is $n$-facet-paired.
\end{defn}

\begin{pro}\label{pro:f|E(p)connected}
Let $f\colon E \to X$ be a simplicial map, where $X$ is $n$-facet-paired. For any simplicial distribution $p \in \sDist(f)$, the restriction $f|_{E(p)}$ is an $n$-facet-paired bundle scenario.
\end{pro}
\begin{proof}
This follows directly from Proposition~\ref{pro:f|Ebundle}.
\end{proof}

\begin{lem}\label{lem:BundlePreserveNDegen}
Let $f\colon E\to X$ be a simplicial map that is discrete over vertices.
Then a simplex $e\in E_n$ is nondegenerate if and only if
$f_n(e)$ is nondegenerate.
\end{lem}

\begin{proof}
If $e$ is degenerate, then the simpliciality of $f$ implies that
$f_n(e)$ is degenerate.

Conversely, suppose that $f_n(e)$ is degenerate. Then there exists
$x\in X_{n-1}$ such that
$
s_i(x)=f_n(e)
$ for some $0\leq i\leq n-1$.
Consider the following commutative diagram:
$$
\begin{tikzcd}
\Delta^n \arrow[r,"e"] \arrow[d,"s^i"'] &
E \arrow[d,"f"] \\
\Delta^{n-1} \arrow[r,"x"'] \arrow[ru,dashed,"e'"] &
X.
\end{tikzcd}
$$
Since $f$ is discrete over vertices, there exists a lift
$e'\colon\Delta^{n-1}\to E$. In particular, $s_i(e')=e$.
Hence $e$ is degenerate.    
\end{proof}

\begin{pro}\label{pro:propertiesofE}
Let $X$ be an $n$-facet-paired simplicial set and let
$f\colon E\to X$ be a simplicial map that is discrete over vertices with finite fibers.
Then
\begin{enumerate}
\item $E$ is an $n$-dimensional simplicial set such that the sets $E_n^\circ$ and
$E_{n-1}^\circ$ are finite;

\item for any nondegenerate simplex $e\in E_n$ and any
$0\leq i\leq n$, the face $d_i(e)$ is also nondegenerate;

\item for any nondegenerate simplex $e\in E_n$,  $i\neq j$ implies $d_i(e)\neq d_j(e)$.
\end{enumerate}
\end{pro}
\begin{proof}
The result follows from Proposition~\ref{pro:facet-paired-properties} and Lemma~\ref{lem:BundlePreserveNDegen}.
\end{proof}

By Propositions~\ref{pro:f|Ebundle} and~\ref{pro:propertiesofE}, it follows that if $p$ is a simplicial distribution on a simplicial map
$f\colon E\to X$, where $X$ is $n$-facet-paired, then the space of supports $E(p)$ satisfies the properties listed in Proposition~\ref{pro:propertiesofE}.

\subsection{The facet matrix of the space of supports}\label{subsec:Afchi}

Now, we associate a matrix to a simplicial distribution whose
measurement space is an $n$-facet-paired simplicial set. We then show that the
dimension of the carrier of the simplicial distribution is equal to the
dimension of the null space of this matrix minus one.

\begin{defn}
Let $X$ be an $n$-facet-paired simplicial set. Define a graph $G_X$ whose vertices are the nondegenerate $n$-simplices of $X$ and whose edges are the nondegenerate $(n-1)$-simplices of $X$. Each edge $\tau$ joins the two unique vertices corresponding to the two $n$-simplices having $\tau$ as a face (see Condition~(3) in Definition~\ref{def:facet-paired}).

A \emph{direction} $\chi$ on $X$ is a choice of orientation for each edge of $G_X$. We write
$$
\sigma \rightarrow \tau \quad \text{if $\sigma$ is the head of the edge $\tau$,}
\qquad
\tau \rightarrow \sigma \quad \text{if $\sigma$ is the tail of the edge $\tau$.}
$$
\end{defn}

\begin{defn}\label{def:Afchi}
Let $\chi$ be a direction on $X$  
and let $f \colon E \to X$ be an $n$-facet-paired bundle scenario. Fix orders for 
$E^\circ_n$ and for $E^\circ_{n-1}$. 
The \emph{facet matrix} 
$
A(f,\chi)\in M_{|E^\circ_{n-1}|\times |E^\circ_n|}(\ZZ)
$
corresponding to $f$ and $\chi$, with rows indexed by $E^\circ_{n-1}$ and columns indexed by $E^\circ_n$, is defined by the following entries:
\begin{equation}\label{eq:Afchi}
A(f,\chi)_{w,e}
=
\begin{cases}
1,&\text{if $w$ is a face of $e$ and  $f_n(e) \rightarrow f_{n-1}(w)$},\\
-1,&\text{if $w$ is a face of $e$  and $f_{n-1}(w) {\rightarrow} f_n(e)$},\\
0,&\text{otherwise}.
\end{cases}
\end{equation}
\end{defn}
In Definition~\ref{def:Afchi}, the sets $E_n^\circ$ and
$E_{n-1}^\circ$ are finite by part (1) of Proposition \ref{pro:propertiesofE}. 

\begin{rem}\label{rem:samedi}
In Definition \ref{def:Afchi}, suppose that $d_j(e)=w$ and $d_i(f_n(e))=f_{n-1}(w)$. Then 
$$ 
d_j(f_n(e))=f_{n-1}(d_j(e))=f_{n-1}(w). 
$$
Since $f_n(e)$ is nondegenerate, by property (3) of Proposition \ref{pro:facet-paired-properties}, we get $i=j$.
\end{rem}

\begin{ex}\label{ex:Achif}
Let $E$ be the simplicial subset of $\Delta^4$ generated by the following nondegenerate $2$-simplices:
$$
E_2^\circ=
\{[0,1,2],[0,1,3],[0,2,3],[1,2,3],
[0,1,4],[0,2,4],[1,2,4]\}.
$$
Consider the $2$-facet-paired bundle scenario $f\colon E\longrightarrow \partial\Delta^3$ (Example~\ref{ex:boundaryofsimplex}), 
defined on vertices by
$$
f_0([0])=[0],\qquad
f_0([1])=[1],\qquad
f_0([2])=[2],\qquad
f_0([3])=[3],\qquad
f_0([4])=[3].
$$
We determine the simplicial map $f$ by its action on $E_0$. In particular,
its values on the generating $2$-simplices are obtained by applying $f_0$
to their vertices. For example,
$$
f_2([0,1,2])=[0,1,2],
\qquad
f_2([1,2,4])=[1,2,3].
$$

We choose $\chi$ so that the corresponding facet matrix $A(f,\chi)$ is
\vspace{-18pt}
    \begin{center}
        $$
        A(f,\chi) =
        \begin{array}{c|ccccccc}
            & [0,1,2] & [0,1,3] & [0,2,3] & [1,2,3] & [0,1,4] & [0,2,4] & [1,2,4]\\
            \hline
            [0,1] & 1 & -1 & 0 & 0 & -1 & 0 & 0 \\{}
            [0,2] & 1 & 0 & -1 & 0 & 0 & -1 & 0 \\{}
            [1,2] & 1 & 0 & 0 & -1 & 0 & 0 & -1 \\{}
            [0,3] & 0 & 1 & -1 & 0 & 0 & 0 & 0 \\{}
            [1,3] & 0 & 1 & 0 & -1 & 0 & 0 & 0 \\{}
            [2,3] & 0 & 0 & 1 & -1 & 0 & 0 & 0 \\{}
            [0,4] & 0 & 0 & 0 & 0 & 1 & -1 & 0 \\{}
            [1,4] & 0 & 0 & 0 & 0 & 1 & 0 & -1 \\{}
            [2,4] & 0 & 0 & 0 & 0 & 0 & 1 & -1 
        \end{array}$$
    \end{center}
\end{ex}

\begin{lem}\label{lem:lin=affin}
Let $p^{(1)},\dots, p^{(k)}$ be simplicial distributions on a simplicial map $f\colon E \to X$.
Then the following
are equivalent:
\begin{itemize}
    \item $p^{(1)},\dots, p^{(k)}$ are affinely independent.
    \item $p^{(1)},\dots, p^{(k)}$ are linearly independent.
\end{itemize}  
\end{lem}
\begin{proof}
Suppose
$
t_1p^{(1)}+\dots+t_kp^{(k)}=0.
$ 
Then, for any simplex $\sigma \in X_n$, we have
\begin{align*}
t_1+\cdots+t_k
&=
t_1\sum_{a\in f_n^{-1}(\sigma)}p_\sigma^{(1)}(a)
+\cdots+
t_k\sum_{a\in f_n^{-1}(\sigma)}p_\sigma^{(k)}(a)\\
&=
\sum_{a\in f_n^{-1}(\sigma)}
\left(
t_1p_\sigma^{(1)}(a)+\cdots+t_kp_\sigma^{(k)}(a)
\right)\\
&=
\sum_{a\in f_n^{-1}(\sigma)}
\left(t_1p_\sigma^{(1)}+\cdots+t_kp_\sigma^{(k)}\right)(a) = 0. \qedhere
\end{align*}
\end{proof}

Let $f\colon E\to X$ be a simplicial map, where $X$ is $n$-facet-paired, and let $p\in\sDist(f)$. Since $p$ is determined by its values on the generating simplices of $X$, and since $p_\sigma\in D(f_n^{-1}(\sigma))$ for every generating $n$-simplex $\sigma$ of $X$, Lemma~\ref{lem:BundlePreserveNDegen} allows us to represent $p$ by the vector
\begin{equation}\label{eq:gconnecandpvector} 
p= \left( p_{f_n(e)}(e) \right)_{e\in E(p)_n^\circ}. 
\end{equation}

\begin{pro}\label{pro:carr=facetmat}
Let $f\colon E\to X$ be a simplicial map, where $X$ is $n$-facet-paired, and let $p\in\sDist(f)$. 
Denote by $g=f|_{E(p)}$ the restriction of $f$ to the space of supports of $p$. 
Then, for any choice of direction $\chi$ on $X$, 
$$ 
\dim(\carr(p)) = \dim\bigl(\Null(A(g,\chi))\bigr)-1. 
$$
\end{pro}
\begin{proof}
First, the matrix {$A:=A(g,\chi)$} is well-defined by Proposition~\ref{pro:f|E(p)connected}. 

Since $p$ can be regarded as a simplicial distribution on $g$, we view it as the vector 
$\left(p_{f_n(e)}(e)\right)_{e \in E(p)^\circ_n}$, as in Equation~(\ref{eq:gconnecandpvector}). 
Moreover, $q\preceq p$ if and only if the defining conditions in Definition~\ref{def:GenRelation} hold on the generating simplices. Thus, such a $q$ can also be regarded as a vector 
$(q_{f_n(e)}(e))_{e \in E(p)^\circ_n}$.

We first prove that
$$
\set{q:\; q\preceq p} \subseteq \Null(A).
$$
Let $q\preceq p$ and let $w\in E(p)^\circ_{n-1}$. 
{Denote by $\sigma_1,\sigma_2$ the unique nondegenerate simplices in $X_n$ that are respectively the head and tail of $g_{n-1}(w)$, satisfying $d_{i_1}(\sigma_1)=d_{i_2}(\sigma_2)=g_{n-1}(w)$.}

Then using Equation (\ref{eq:Afchi}) and Remark \ref{rem:samedi}, we have
$$
\begin{aligned}
(Aq)_w&=\sum_{e \in E(p)^\circ_n} A_{w,e} q_{f_n(e)}(e) \\ 
&=
\sum_{e \in E(p)^\circ_n} \left(\begin{gathered} 
\sum_{\substack{e:\; d_{i_1}(e)=w\\ g_n(e)\rightarrow g_{n-1}(w)}} q_{f_n(e)}(e)-
\sum_{\substack{e:\; d_{i_2}(e)=w\\ g_{n-1}(w)\rightarrow g_n(e)}} q_{f_n(e)}(e) \end{gathered}\right)\\
&=\left(\begin{gathered} 
\sum_{e \in E(p)^\circ_n} \;\,\sum_{\substack{e:\;  d_{i_1}(e)=w \\ \sigma_1\rightarrow g_{n-1}(w)}} q_{\sigma_1}(e)\end{gathered}\right) -
\left(\begin{gathered} 
\sum_{e \in E(p)^\circ_n} \;\,\sum_{\substack{e:\;  d_{i_2}(e)=w\\ g_{n-1}(w) \rightarrow \sigma_2}} q_{\sigma_2}(e) \end{gathered}\right) \\
&=\left(\begin{gathered} 
\sum_{e \in f_n^{-1}(\sigma_1) \text{ and } d_{i_1}(e)=w } q_{\sigma_1}(e)\end{gathered}\right)
-\left(\begin{gathered} 
\sum_{e \in f_n^{-1}(\sigma_2) \text{ and } d_{i_2}(e)=w} q_{\sigma_2}(e)\end{gathered}\right) \\
&=D(d_{i_1})(q_{\sigma_1})(w)-D(d_{i_2})(q_{\sigma_2})(w) \\
&=q_{d_{i_1}(\sigma_1)}(w)-q_{d_{i_2}(\sigma_2)}(w) =0
\end{aligned}
$$

Hence $q \in \Null(A)$. By Proposition~\ref{pro:yleqxVsuppcarr}, we conclude that
\begin{equation}\label{eq:carrsubseteqnull}
\carr(p) \subseteq \Null (A)
\end{equation}
Suppose that $\dim(\Null(A))=r$. Then, by
Lemma~\ref{lem:lin=affin} and~\eqref{eq:carrsubseteqnull}, there cannot be
$r+1$ affinely independent points in $\carr(p)$. Therefore,
\begin{equation}\label{eq:dimcarrleqrnew}
\dim(\carr(p))\leq r-1,
\end{equation}
see Definition~\ref{def:dimF}.

Let $k=|E(p)^\circ_n|$. We extend $p$ to a basis of $\Null(A)$:
$$
p,x^{(2)},\dots,x^{(r)}\in\RR^k.
$$
For every $2\leq i\leq r$, there exists $t_i$ such that
$(x^{(i)}+t_i p)_j\geq 0$ for all $1\leq j\leq k$. Hence we may construct an alternative basis of $\Null(A)$ with all coordinates nonnegative:
$$
p,\,y^{(2)}=x^{(2)}+t_2p,\,\dots,\,y^{(r)}=x^{(r)}+t_rp
$$

Fix $y=y^{(i)}$ for some $2\leq i\leq r$. For every nondegenerate $\sigma\in X_n$, define
$$
T_{y,\sigma}=\sum_{e\in g_n^{-1}(\sigma)}y_e.
$$
Let $\tau\in X_{n-1}$ be nondegenerate with
$\sigma_1\rightarrow\tau\rightarrow\sigma_2$. Then
\begin{equation}\label{eq:prenonsignaling}
\begin{aligned}
T_{y,\sigma_1}-T_{y,\sigma_2}
&=\sum_{e\in g_n^{-1}(\sigma_1)}y_e
-\sum_{e\in g_n^{-1}(\sigma_2)}y_e\\
&=\sum_{w\in g_{n-1}^{-1}(\tau)}
\left(\begin{gathered} 
\sum_{\substack{e\in g_n^{-1}(\sigma_1)\\ w\text{ is a face of }e}}y_e
-
\sum_{\substack{e\in g_n^{-1}(\sigma_2)\\ w\text{ is a face of }e}}y_e
\end{gathered}\right)\\
&=\sum_{w\in g_{n-1}^{-1}(\tau)} \;
\sum_{e\in E(p)^\circ_n}A_{w,e}y_e\\
&=\sum_{w\in g_{n-1}^{-1}(\tau)}(Ay)_w=0.
\end{aligned}
\end{equation}
Thus $T_{y,\sigma_1}=T_{y,\sigma_2}$. Since $X$ is $n$-facet-paired, any two nondegenerate $n$-simplices are connected through a chain of shared $(n-1)$-faces. Hence $T_{y,\sigma}$ is independent of $\sigma$, and we denote this common value by $T_y$.

Since $y\neq 0$ and $y_e\geq 0$ for every $e\in E(p)_n^\circ$, we have $T_y>0$. We now define a simplicial distribution $q\in\sDist(f)$ by
\begin{equation}\label{eq:defqnew}
q_{\sigma}(e)=
\begin{cases}
\dfrac{y_e}{T_y}, & \text{if $e\in E(p)^\circ_n$ and $f_n(e)=\sigma$,}\\
0, & \text{otherwise,}
\end{cases}
\end{equation}
at every generating simplex $\sigma\in X_n$.

First, $q_\sigma$ is normalized, since
$$
\sum_{e\in f_n^{-1}(\sigma)}q_\sigma(e)
=
\frac{1}{T_y}\sum_{e\in g_n^{-1}(\sigma)}y_e
=
\frac{T_y}{T_y}=1.
$$
To show that $q$ defines a simplicial distribution, since $X$ is $n$-facet-paired, it suffices to verify compatibility on the common nondegenerate $(n-1)$-faces.
Let $w\in E(p)_{n-1}^\circ$ satisfy
$\sigma_1\rightarrow f_{n-1}(w)\rightarrow\sigma_2$, where $d_{i_1}(\sigma_1)=f_{n-1}(w)=d_{i_2}(\sigma_2)$. Then
$$
\begin{aligned}
&\sum_{\substack{e\in g_n^{-1}(\sigma_1)\\ d_{i_1}(e)=w}}q_{\sigma_1}(e)
-
\sum_{\substack{e\in g_n^{-1}(\sigma_2)\\ d_{i_2}(e)=w}}q_{\sigma_2}(e)\\
&=
\frac{1}{T_y}
\left(\begin{gathered}
\sum_{\substack{e\in g_n^{-1}(\sigma_1)\\ d_{i_1}(e)=w}}y_e
-
\sum_{\substack{e\in g_n^{-1}(\sigma_2)\\ d_{i_2}(e)=w}}y_e
\end{gathered}\right)\\
&=
\frac{1}{T_y}\sum_{e\in E(p)^\circ_n}A_{w,e}y_e
=
\frac{1}{T_y}(Ay)_w=0.
\end{aligned}
$$
Thus, we obtain a basis of $\Null(A)$ consisting of simplicial distributions:
$$
p,z^{(2)},\dots,z^{(r)},
\qquad
z^{(i)}=\frac{y^{(i)}}{T_{y^{(i)}}}.
$$
By~\eqref{eq:defqnew}, we have $z^{(i)}\preceq p$ for every $2\leq i\leq r$. Hence, by Proposition~\ref{pro:yleqxVsuppcarr}, $z^{(i)}\in\carr(p)$
for all $2\leq i\leq r$. Therefore, $\carr(p)$ contains $r$ linearly independent, and hence affinely independent, simplicial distributions. Combining this with~\eqref{eq:dimcarrleqrnew}, we conclude that
\begin{equation*}
\dim(\carr(p))=r-1.\qedhere
\end{equation*}
\end{proof}


\section{The homology of the space of supports}\label{sec:Homolgcar}

In this section, we establish a connection between the dimension of the carrier of a simplicial distribution and 
the homology of its support. 
We begin by recalling the basic notions of simplicial homology and normalized chain complexes that will be used throughout the section.

\subsection{The homology of simplicial sets}

Homology of simplicial sets is {a} fundamental homotopy {invariant} in algebraic topology 
(see e.g.\ \cite[Section I.1]{goerss2009simplicial}).

We begin by fixing some notations for chain complexes. Let $(C_\bullet,\partial_\bullet)$ be a chain complex of abelian groups. The groups of \emph{$n$-cycles} and \emph{$n$-boundaries} are defined, respectively, by
$$ 
Z_n(C_\bullet):=\ker(\partial_n) \quad \text{and} \quad B_n(C_\bullet):=\operatorname{im}(\partial_{n+1}). 
$$ 
Since $\partial_n\circ\partial_{n+1}=0$, we have $B_n(C_\bullet)\subseteq Z_n(C_\bullet)$. The $n$-th homology group of $C_\bullet$ is 
$$
H_n(C_\bullet) := \frac{Z_n(C_\bullet)}{B_n(C_\bullet)}. 
$$

Let $X$ be a simplicial set.
The \emph{group of $n$-chains} of $X$ is the free abelian group generated by the set of $n$-simplices:
$$
C_n(X) := \ZZ[{X_n}].
$$

The boundary homomorphism $\partial_n=\partial^X_n\colon C_n(X) \to C_{n-1}(X)$
is defined on {the} generating simplices $\sigma \in X_n$ by
$$
\partial_n(\sigma)
=
\sum_{i=0}^{n} (-1)^i \, d_i(\sigma),
$$

The simplicial identities imply that $\partial_{n-1} \circ \partial_n = 0$,
so $(C_\bullet(X), \partial_\bullet)$ is a chain complex, called the
\emph{simplicial chain complex} of $X$. Then, the \emph{$n$-th homology group of $X$} is defined by
$$
H_n(X) 
=\frac{Z_n\left(C_\bullet(X)\right)}{B_n\left(C_\bullet(X)\right)}=
\ker(\partial_n) \big/ \operatorname{im}(\partial_{n+1}).
$$
For our purposes, {we will} work with the normalized chain complex, which has the same homology as the simplicial chain complex. {This relationship} is discussed in detail in \cite[Section III.2]{goerss2009simplicial} and \cite[Proposition 16.2, Corollary 22.3]{may1992simplicial}.
\begin{defn}
Let $X$ be a simplicial set.
The \emph{degenerate subgroup} $D_n(X)\subseteq C_n(X)$ is the subgroup
generated by all degenerate $n$-simplices of $X$. Equivalently,
$$
D_n(X)
=
\ZZ(\set{
s_i(\sigma) :\;
\sigma\in X_{n-1},\ 0\leq i\leq n-1})
$$

The \emph{normalized chain group} of $X$ is defined by
$$
N_n(X):=C_n(X)/D_n(X).
$$
The boundary maps on $C_\bullet(X)$ induce boundary maps
$$
\partial_n\colon N_n(X)\longrightarrow N_{n-1}(X),
$$
and hence $N_\bullet(X)$ is a chain complex, called the
\emph{normalized chain complex} of $X$.
\end{defn}

Let $X$ be a simplicial set. {The} nondegenerate
$n$-simplices form a basis for the normalized chain group $N_n(X)$.
In particular,
$$
N_n(X)
\cong
\bigoplus_{\substack{\sigma\in X_n\\
\sigma\text{ nondegenerate}}}
\mathbb{Z}\,[\sigma].
$$
Consequently, every element $c\in N_n(X)$ {can be uniquely expressed as}
$$
c
=
\sum_{\substack{\sigma\in X_n\\
\sigma\text{ nondegenerate}}}
\alpha_\sigma[\sigma],
\qquad
\alpha_\sigma\in\mathbb{Z},
$$
where only finitely many coefficients $\alpha_\sigma$ are nonzero.

If $X$ is an $n$-dimensional simplicial set, then $N_{n+1}(X)=0$. Consequently, $B_n(N_\bullet X)=0$.
Therefore,
\begin{equation}\label{eq:Hn(X)=ker}
H_n(X)=\frac{Z_n(N_\bullet X)}{B_n(N_\bullet X)}
\cong
Z_n(N_\bullet X).
\end{equation}


\subsection{The dimension of the carrier in terms of homology}

We now relate the dimension of the carrier of a simplicial distribution to the top-dimensional homology of its support. To this end, we introduce a matrix associated with the normalized boundary map and compare it with the facet matrix $A(f,\chi)$ introduced in Section \ref{subsec:Afchi}.

\begin{defn}\label{def:BEEEE}
Let $E$ be an $n$-dimensional simplicial set such that
$E_n^\circ$ and $E_{n-1}^\circ$ are finite, and suppose that for every
nondegenerate simplex $e\in E_n^\circ$ and every $0\leq i\leq n$, the face
$d_i(e)$ is also nondegenerate.

Fix orderings of $E_n^\circ$ and $E_{n-1}^\circ$. We define the {boundary} matrix
$B(E)\in M_{|E_{n-1}^\circ|\times |E_n^\circ|}(\ZZ)$
by setting, for every $(w,e)\in E_{n-1}^\circ\times E_n^\circ$,
$$
B(E)_{w,e}
=
\sum_{\substack{0\leq k\leq n\\ d_k(e)=w}}
(-1)^k,
$$
where the sum is understood to be zero if $w$ is not a face of $e$.

Equivalently, $B(E)$ is the matrix representation of the normalized
boundary map
$$
\partial_n\colon N_n(E)\longrightarrow N_{n-1}(E)
$$
with respect to the ordered bases given by the nondegenerate
$n$-simplices and $(n-1)$-simplices of $E$.
\end{defn}

\begin{defn}\label{def:Bfff}    
Let $f \colon E \to X$ be an $n$-facet-paired bundle scenario. 
We define $B(f)$ to be $B(E)$.   
\end{defn}
The conditions on $E$ required in Definition~\ref{def:BEEEE} are satisfied by parts~(1) and~(2) of Proposition~\ref{pro:propertiesofE}. Moreover, $E_n^\circ$ and $E_{n-1}^\circ$ are finite since $X$ is finitely generated and the fibers of $f$ are finite. Hence the matrix $B(f)$ is well defined.

 \begin{ex}\label{ex:Bffff}
Let $f$ be as in Example \ref{ex:Achif}. Then 
    \begin{center}
        $$B(f) =
        \begin{array}{c|ccccccc}
            & [0,1,2] & [0,1,3] & [0,2,3] & [1,2,3] & [0,1,4] & [0,2,4] & [1,2,4]\\
            \hline
            [0,1] & 1 & 1 & 0 & 0 & 1 & 0 & 0 \\{}
            [0,2] & -1 & 0 & 1 & 0 & 0 & 1 & 0 \\{}
            [1,2] & 1 & 0 & 0 & 1 & 0 & 0 & 1 \\{}
            [0,3] & 0 & -1 & -1 & 0 & 0 & 0 & 0 \\{}
            [1,3] & 0 & 1 & 0 & -1 & 0 & 0 & 0 \\{}
            [2,3] & 0 & 0 & 1 & 1 & 0 & 0 & 0 \\{}
            [0,4] & 0 & 0 & 0 & 0 & -1 & -1 & 0 \\{}
            [1,4] & 0 & 0 & 0 & 0 & 1 & 0 & -1 \\{}
            [2,4] & 0 & 0 & 0 & 0 & 0 & 1 & 1 \\
        \end{array}$$
    \end{center}
 \end{ex}

\begin{defn}
We say that an $n$-dimensional simplicial set $X$ is \emph{orientable} if
there exists a choice of signs 
$\epsilon_\sigma\in\{+1,-1\}$, $\sigma\in X_n^\circ$, such that
$$
\sum_{\sigma \in X_n^\circ}\epsilon_{\sigma}\sigma
$$
is an $n$-cycle in the normalized chain complex $N_\bullet(X)$.
\end{defn}

\begin{pro}\label{pro:rank-A(f)-rank-B(f)}

Let $X$ be an orientable $n$-dimensional simplicial set, let $f\colon E\to X$ be an $n$-facet-paired bundle scenario, and let $\chi$ be a direction on $X$. Then
$$
{\rank}A(f,\chi)={\rank}B(f)
$$
\end{pro}

\begin{proof}
{Fix orderings of $E^\circ_n$ and $E^\circ_{n-1}$, thereby inducing ordered bases of $N_n(E)$ and $N_{n-1}(E)$,} respectively. Then $B(f)$ is the matrix {representation}
of the normalized boundary map
$\partial^E_n\colon N_n(E)\to N_{n-1}(E)$
with respect to these bases. To prove the result, we will choose signs
$a_e,b_\tau\in\set{1,-1}$ for all
$e\in E^\circ_n$ and $\tau\in E^\circ_{n-1}$,
{based} on the direction $\chi$, such that the matrix {representation} of
$\partial^E_n$ with respect to the ordered bases
$\{a_e e\}_{e\in E^\circ_n}$ and
$\{b_\tau\tau\}_{\tau\in E^\circ_{n-1}}$,
with the original orderings preserved, is precisely $A(f,\chi)$.

By orientability, there exist signs
$$
\{\epsilon_\sigma\}_{\sigma \in X_n^{\circ}}\in\{+1,-1\}
$$
such that $\partial_n(\sum_{\sigma\in X_n^\circ}\epsilon_\sigma \sigma) = 0$.
    For each $e\in E_n^\circ$, since $f$ is discrete over vertices, {$f_n(e)\in X^\circ_n$ by Lemma \ref{lem:BundlePreserveNDegen}.} So we define
$$
a_e = \epsilon_{f_n(e)}.
$$

    For each $\tau\in E_{n-1}^\circ$, {since} $X$ is $n$-facet-paired, there exist exactly two distinct simplices $\sigma_1,\sigma_2\in X_n^\circ$ {such that} $d_{i_1}(\sigma_1) = d_{i_2}(\sigma_2) =f_{n-1}(\tau)$ for some $i_1,i_2$. Without loss of generality, we let $\sigma_1$ be the head and let $\sigma_2$ be the tail of $f_{n-1}(\tau)$ in $\chi$, and then define
    $$
    b_\tau = (-1)^{i_1}\epsilon_{\sigma_1} = (-1)^{i_2+1}\epsilon_{\sigma_2}.
    $$

    To see that $b_\tau$ is well defined, {since} $\partial_n(\sum_{\sigma\in X_n^\circ}\epsilon_\sigma \sigma) = 0$ and $\sigma_1,\sigma_2$ are the only nondegenerate simplices that have 
    $f_{n-1}(\tau)$ as a face, we obtain that
    $$
    0 =\partial_n (\epsilon_{\sigma_1}\sigma_1 + \epsilon_{\sigma_2}\sigma_2 +\dots) 
    = \left((-1)^{i_1}\epsilon_{\sigma_1} + (-1)^{i_2}\epsilon_{\sigma_2}\right)f_{n-1}(\tau) +\dots.
    $$
Then $(-1)^{i_1}\epsilon_{\sigma_1} + (-1)^{i_2}\epsilon_{\sigma_2} = 0$. Hence $b_\tau$ is well-defined.


Now, let $B'(f)$ be the matrix representation of $\partial^E_n$ with respect to the ordered bases $\{a_ee\}_{e\in E^\circ_n}$ and $\{b_\tau\tau\}_{\tau\in E_{n-1}^\circ}$. Fix $e\in E^\circ_n$ and $\tau \in E^\circ_{n-1}$. We show that $B'(f)_{\tau,e} = A(f,\chi)_{\tau,e}$. 

If for some $k$, we have $d_k(e)=\tau$, then 
$$
\begin{aligned}
\partial^{E}_n(a_e e)&=a_e\partial^E_n(e) \\
&=a_e \sum_{i=0}^{n} (-1)^i \, d_i(e)  \\
&=a_e(-1)^k\tau +  \sum_{i\neq k} a_e (-1)^i \, d_i(e) \\
&=a_e(-1)^k(b_{\tau})^2\tau +  \sum_{i\neq k} a_e (-1)^i \, d_i(e) \\
&=a_e b_\tau(-1)^k b_{\tau}\tau +  \sum_{i\neq k} a_e (-1)^i \, d_i(e) \\
\end{aligned}
$$
By part (3) of Proposition \ref{pro:propertiesofE}, 
there is no $i \neq k$ with $d_i(e)=\tau$. As a result,
$$
\begin{aligned}
B'(f)_{\tau,e} = a_e b_{\tau} (-1)^k = \begin{cases}
\epsilon_{f_n(e)} (-1)^k \epsilon_{f_n(e)} (-1)^k=1    & \text{if } f_n(e) \text{ is the head of the edge } f_{n-1}(\tau), \\
\epsilon_{f_n(e)} (-1)^{k+1} \epsilon_{f_n(e)} (-1)^k=-1    & \text{if } f_n(e) \text{ is the tail of the edge } f_{n-1}(\tau). \\    
\end{cases}
\end{aligned}
$$
If $\tau$ is not a face of $e$, then clearly $B'(f)_{\tau,e}=0$.

So we proved that $B'(f)_{\tau,e}$ coincides with $A(f,\chi)_{\tau,e}$ according to Definition \ref{def:Afchi}.
\end{proof}

\begin{ex}
The matrices $A(f,\chi)$ and $B(f)$ of Examples \ref{ex:Achif} and \ref{ex:Bffff} both have rank $5$.   
\end{ex}

\begin{thm}\label{thm:Dimcar=Homolgy}
Let $X$ be an orientable, $n$-facet-paired simplicial set, and let $f\colon E \to X$ be a simplicial map. For any simplicial distribution 
$p$ in $\sDist(f)$, we have 
$$
\dim(\carr(p))= \rank\left( H_n(E(p))\right)-1
$$
\end{thm}
\begin{proof}
Fix a direction $\chi$ for $X$. By Proposition \ref{pro:f|E(p)connected}, $f|_{E(p)}$ is an $n$-facet-paired scenario. Then,
by Propositions \ref{pro:carr=facetmat}, \ref{pro:rank-A(f)-rank-B(f)}, and Equation (\ref{eq:Hn(X)=ker}), 
\begin{equation*}
\begin{aligned}[b]
\dim(\carr(p))&=\dim \left(\Null(A(f|_{E(p)},\chi))\right)- 1\\
&= \dim \left(\Null(B(f|_{E(p)}))\right)- 1 \\
&=\rank(\ker \partial_n^{E(p)})-1 \\
&=\rank\left( H_n(E(p))\right)-1.
\end{aligned}
\qedhere
\end{equation*}
\end{proof}

The following examples demonstrate the necessity of the $n$-facet-paired hypothesis.
\begin{ex}
Let $p$ be the simplicial distribution on
$f_{\Delta^1,2}\colon E(\Delta^1,2)\to \Delta^1$
(see Definition~\ref{def:BXm}), defined by
\begin{equation}\label{eq:distoneedge}
p_{[0,1]}
=
\begin{pmatrix}
1/2 & 0\\
0 & 1/2
\end{pmatrix}.
\end{equation}
Then we have the following bundle scenario
$f_{\Delta^1,2}|_{E(p)}\colon E(p)\to \Delta^1$:

\begin{figure}[h!] 
\begin{center}
\includegraphics[width=0.4\textwidth]{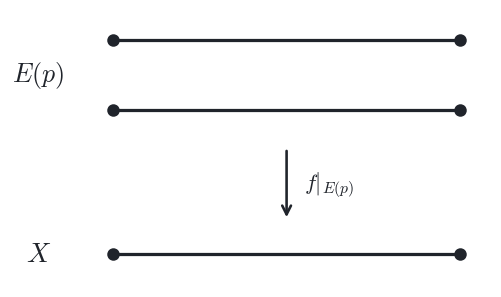}
\caption{The restriction of $f_{\Delta^1,2}$ to the space of supports of the simplicial distribution defined in Equation (\ref{eq:distoneedge}).} 
        \label{fig:towedgessupport}
\end{center}
\end{figure}  
The simplicial set $\Delta^1$ is not $1$-facet-paired, since each of
the $0$-simplices $[0]$ and $[1]$ is a face of only one nondegenerate
$1$-simplex, namely $[0,1]$. Indeed,
$$
\carr(p)=\Conv\set{\begin{pmatrix}
1 & 0\\
0 & 0
\end{pmatrix}, \begin{pmatrix}
0 & 0\\
0 & 1
\end{pmatrix}},   \text{  so } \dim(\carr(p))=1,
$$
but
$$
\rank\left(H_1(E(p))\right)-1
=
0-1
=
-1.
$$
\end{ex}

\begin{ex}
Consider the simplicial distribution $p$ of Example~\ref{ex:233bellvertex}. Note that
$$
|E(p)_0|=3|sB(2,3)_0|=18.
$$
Moreover, every nonzero entry of each $p_{x_i,y_j}$ corresponds to an edge of
$E(p)_1^\circ$. Since $E(p)$ is connected, the Euler characteristic formula
for connected graphs \cite{hatcher2002algebraic} gives
$$
\rank H_1(E(p))=32-18+1=15.
$$
On the other hand, $p$ is a vertex, so $\dim(\carr(p))=0$.
\end{ex}

\begin{ex}
Let $X$ be the simplicial set illustrated at the bottom of Figure \ref{fig:conterex2}. 
Consider the simplicial distribution $p$ on
$f_{X,2}\colon E(X,2)\to X$ defined as follows: 
\begin{equation}\label{eq:secondcounterex}
\begin{aligned}
p_{[1,2,4]},p_{[2,4,5]},p_{[2,3,5]},p_{[3,5,6]},p_{[1,3,6]},p_{[1,4,6]}
&=
\begin{pmatrix}
1/2 & 0\\
0 & 0
\end{pmatrix}  \qquad
\begin{pmatrix}
0 & 0\\
0 & 1/2
\end{pmatrix},
\\[1ex]
p_{[0,1,2]},p_{[0,1,3]},p_{[0,2,3]},p_{[0,4,5]},p_{[0,4,6]},p_{[0,5,6]}
&=
\begin{pmatrix}
0 & 0\\
1/2 & 0
\end{pmatrix} \qquad
\begin{pmatrix}
0 & 1/2\\
0 & 0
\end{pmatrix}.
\end{aligned}
\end{equation}
See Example \ref{ex:simpdisanydim}.
The bundle scenario
$f_{X,2}|_{E(p)}\colon E(p)\to X$ is shown in Figure \ref{fig:conterex2}. 

\begin{figure}[h!] 
\begin{center}
\includegraphics[width=0.44\textwidth]{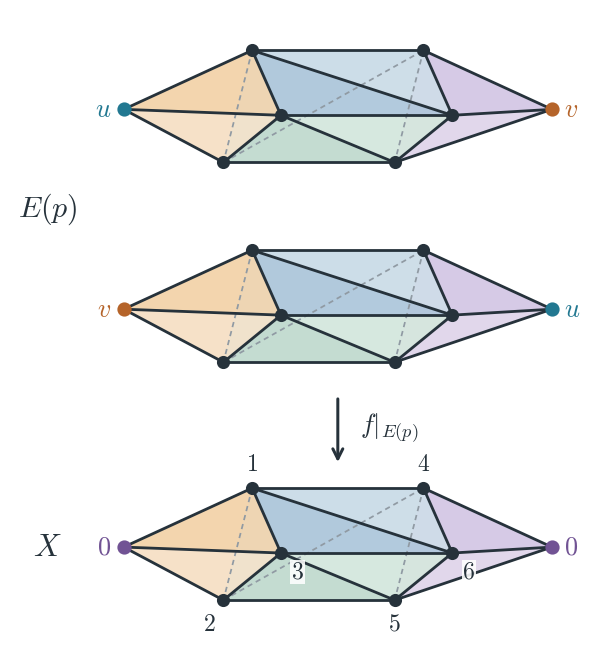}
\caption{The corresponding bundle scenario for the simplicial distribution defined in Equation (\ref{eq:secondcounterex}). {Vertices} of the same non-black color are identified.} 
    \label{fig:conterex2}
\end{center}
\end{figure}  
The bundle scenario $f_{X,2}|_{E(p)}$ is strongly connected
in the sense of \cite[Definition 4.11]{kharoof2026possibilistic}.
Thus, by \cite[Theorem 4.13]{kharoof2026possibilistic},
$p$ is a vertex of $\sDist(f_{X,2})$.
On the other hand,
$\rank\left(H_2(E(p))\right)=2$. Hence,
$$
\dim(\carr(p))=0
\quad\text{but}\quad
\rank\left(H_2(E(p))\right)-1=1.
$$
Thus, the $2$-facet-paired hypothesis cannot be replaced
by connectivity of the nondegenerate $2$-simplices alone.
\end{ex}

\section{The algebraic characterization of the dimension of the carrier}\label{sec:strongconnct}

In this section, we refine the algebraic characterization of the
carrier dimension by reducing the size of the matrix $A(f,\chi)$
introduced in Section~\ref{subsec:Afchi}.
The main idea is to decompose the space of supports into strongly
connected components and encode the relations between these
components in a smaller matrix. 

\subsection{The strong connectivity relation}

We introduce a connectivity relation on the nondegenerate
$n$-simplices of the source of a simplicial map.
This is a restricted version of the strong connectivity relation
introduced in \cite[Definition 4.11]{kharoof2026possibilistic},
where simplices of arbitrary dimensions may be connected along
common faces. Here, we consider only nondegenerate $n$-simplices
connected along nondegenerate $(n-1)$-faces.
The resulting strongly connected components will be used
to construct the reduced matrix.

\begin{defn}\label{def:strong-connectivity}
Let $f\colon E\to X$ be a simplicial map and $n\geq 1$. Consider three simplices
$e_1\in E^\circ_{n}$, $e_2\in E^\circ_{n}$, and $e\in E^\circ_{n-1}$. 
We write $e_1 \sim_{e,f} e_2$ (or shortly $e_1 \sim_{e} e_2$) if 
there exist $i_1$ and $i_2$ such that
$$
d_{i_1}(e_1)=e=d_{i_2}(e_2), 
$$
and the following lifting problems admit unique solutions:
\begin{equation}\label{dia:e_1xe_2}
\begin{tikzcd}[column sep=huge, row sep=large]
\Delta^{n-1} \arrow[d,hook,"d^{i_1}"] \arrow[r,"e"] & E \arrow[d,"f"]
&
\Delta^{n-1} \arrow[d,hook,"d^{i_2}"] \arrow[r,"e"] & E \arrow[d,"f"] \\
\Delta^{n} \arrow[ru,dashed] \arrow[r,"f_{n}(e_1)"] & X
&
\Delta^{n}\arrow[ru,dashed] \arrow[r,"f_{n}(e_2)"] & X .
\end{tikzcd}
\end{equation}
The unique solutions are required to be the simplicial maps
$$
\Delta^{n}\xrightarrow{\,e_1\,}E
\qquad\text{and}\qquad
\Delta^{n}\xrightarrow{\,e_2\,}E.
$$
We say that $e_1$ and $e_2$ are \emph{$f$-strongly connected} if $e_1=e_2$ or 
there exist simplices $e'_1,\dots,e'_{\ell}\in E^\circ_n$ and 
$\tilde{e}_1,\dots,\tilde{e}_{\ell+1}\in E^\circ_{n-1}$ such that
$$
e_1\sim_{\tilde{e}_1} e'_1\sim_{\tilde{e}_2} \cdots \sim_{\tilde{e}_{\ell}} e'_\ell \sim_{\tilde{e}_{\ell+1}} e_2.
$$
This defines an equivalence relation on $E^\circ_n$, denoted by $\sim_f$.
\end{defn}

\begin{ex}\label{ex:conncompvertex}
Consider the $(3,2,2)$ Bell scenario, consisting of three
observers, each with two measurements and two outcomes per
measurement. This scenario is represented by the simplicial map
$f_{sB(3,2),2}$. The measurement space $sB(3,2)$ is a $2$-dimensional
simplicial set generated by the elements of
$A_1\times A_2\times A_3$, where
$$
A_1=\{x_1,x_2\},\quad
A_2=\{y_1,y_2\},\quad
A_3=\{z_1,z_2\}.
$$

We write the vertex number (39) in \cite[Table 1]{pironio2011extremal} as a simplicial distribution $p$ on $f_{sB(3,2),2}$: 

\noindent The distributions $p_{x_1,y_1,z_1}$, $p_{x_1,y_1,z_2}$, and $p_{x_1,y_2,z_1}$ are given by 
$$
    \begin{pmatrix}
    1/3 & 0 \\
    0 & 1/6
    \end{pmatrix}
    \qquad
    \begin{pmatrix}
    0 & 1/3 \\
    1/6 & 0
    \end{pmatrix}.
    $$
The distribution $p_{x_1,y_2,z_2}$ is given by
    $$
    \begin{pmatrix}
    0 & 1/3 \\
    1/6 & 0
    \end{pmatrix}
    \qquad
    \begin{pmatrix}
    1/3 & 0 \\
    0 & 1/6
    \end{pmatrix}
    $$
The distribution $p_{x_2,y_1,z_1}$ is given by
    $$
    \begin{pmatrix}
    1/3 & 0 \\
    0 & 1/6
    \end{pmatrix}
    \qquad
    \begin{pmatrix}
    0 & 1/6 \\
    1/6 & 1/6
    \end{pmatrix}.
    $$
The distributions $p_{x_2,y_1,z_2}$ and $p_{x_2,y_2,z_1}$ are given by
    $$
    \begin{pmatrix}
    1/6 & 1/6 \\
    1/6 & 0
    \end{pmatrix}
    \qquad
    \begin{pmatrix}
    1/6 & 0 \\
    0 & 1/3
    \end{pmatrix}.
    $$
The distribution $p_{x_2,y_2,z_2}$ is given by
    $$
    \begin{pmatrix}
    1/6 & 1/6 \\
    0 & 1/6
    \end{pmatrix}
    \qquad
    \begin{pmatrix}
    1/6 & 0 \\
    1/6 & 1/6
    \end{pmatrix}.
    $$
The strongly connected components of
$f|_{E(p)}$ are illustrated in Figure~\ref{fig:strongcomp2}.

\begin{figure}[h!] 
    \begin{center}
        \includegraphics[width=0.9\textwidth]{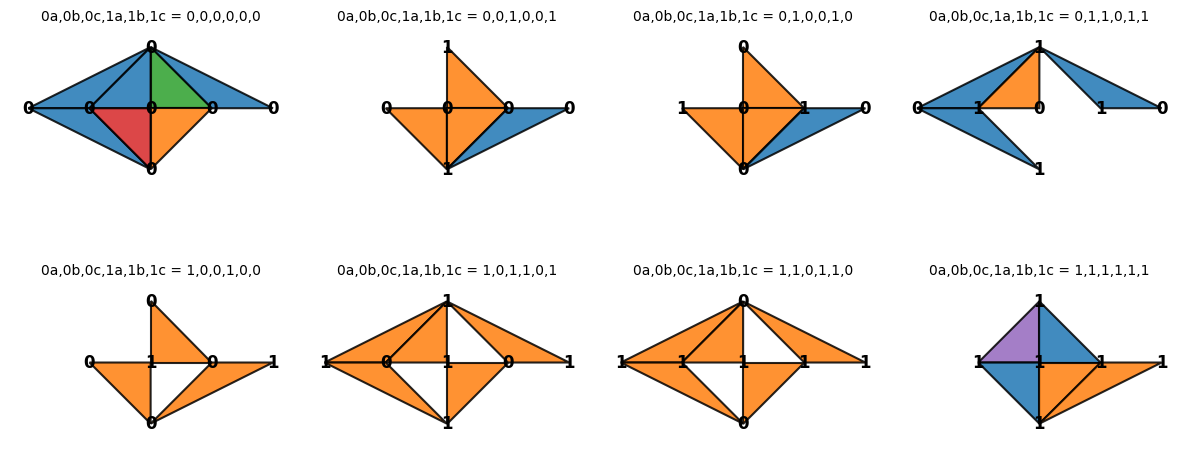}
        \caption{The strongly connected components corresponding to the
simplicial distribution in Example \ref{ex:conncompvertex}}
        \label{fig:strongcomp2}
    \end{center}
\end{figure}  
    
\end{ex}

\subsection{The reduced facet matrix}

We now introduce a reduced version of the facet matrix, defined in Section \ref{subsec:Afchi}, 
using the strongly connected components of $f$.
The resulting matrix records the relations between distinct
strongly connected components along their common
$(n-1)$-faces.

\begin{defn}\label{def:thematrixX(p)}
Let $f\colon E\to X$ be an $n$-facet-paired bundle scenario,
and let
$$
E^{(1)},\dots,E^{(k)}
$$
denote its strongly connected components.

Let $w_1,\dots,w_N$ be the nondegenerate $(n-1)$-simplices
of $E$ that are faces of nondegenerate $n$-simplices
belonging to distinct strongly connected components.

Given a direction $\chi$ on $X$, we define the
\emph{reduced facet matrix}   
$
C(f,\chi)\in M_{N\times k}(\ZZ)
$
by
$$
C(f,\chi)_{i,j}
=
\sum_{e\in E^{(j)}} A(f,\chi)_{w_i,e}.
$$
See Definition \ref{def:Afchi}.

%
\end{defn}
%

\begin{ex}
Consider the measurement space $sB(2,2)$ and the following
simplicial distribution $p$ on $f_{sB(2,2),2}$
(see Example~\ref{ex:CHSHscenario}):
\[
\scalebox{0.85}{$
\begin{tikzcd}[ampersand replacement=\&]
{v_1} \arrow[rrr, "{\begin{pmatrix}11/16 & 0 \\ 0 & 5/16\end{pmatrix}}"]
\& \& \&
{v_2} \arrow[ddd, "{\begin{pmatrix}1/16 & 5/8 \\ 0 & 5/16\end{pmatrix}}"] \\
\& \& \& \\
\& \& \& \\
{v_4} \arrow[uuu, "{\begin{pmatrix}11/16 & 1/4 \\ 0 & 1/16\end{pmatrix}}"]
\& \& \&
{v_3} \arrow[lll, "{\begin{pmatrix}0 & 1/16 \\ 15/16 & 0\end{pmatrix}}"]
\end{tikzcd}
$}
\]
Let $g=f|_{E(p)}\colon E(p)\to sB(2,2)$ be the
restriction of $f$ to the space of supports of $p$.
The map $g$ has six strongly connected components,
illustrated in Figure~\ref{fig:strongcomp-example}.

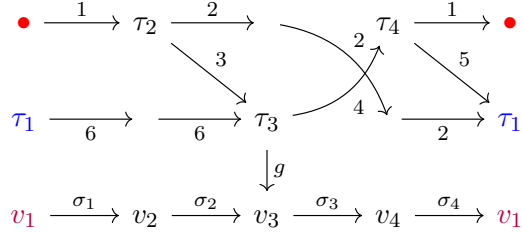
\begin{figure}[h!]
\centering
\begin{tikzcd}[row sep=2em]
{\color{red}\bullet} \arrow[r, "1"]
& {\tau_2} \arrow[r, "2"] \arrow[rd, "3"]
& {} \arrow[rd, "2", bend left]
& {\tau_4} \arrow[r, "1"] \arrow[rd, "5"]
& {\color{red}\bullet} \\
{\color{blue}\tau_1} \arrow[r, "6"']
& {} \arrow[r, "6"']
& {\tau_3} \arrow[ru, "4"', bend right]
 \arrow[d, shorten <=3pt, shorten >=3pt,"g"]
& {} \arrow[r, "2"']
& {\color{blue}\tau_1} \\
{\color{purple}v_1} \arrow[r, "\sigma_1"]
& {v_2} \arrow[r, "\sigma_2"]
& {v_3} \arrow[r, "\sigma_3"]
& {v_4} \arrow[r, "\sigma_4"]
& {\color{purple}v_1}
\end{tikzcd}
\caption{
The numbers $1,\ldots,6$ label the strongly connected
components, and vertices with the same non-black color are identified.}
\label{fig:strongcomp-example}
\end{figure}

Choose a direction $\chi$ on $G_{sB(2,2)}$ given by
$$
\sigma_1\rightarrow v_1\rightarrow\sigma_4,
\qquad
\sigma_2\rightarrow v_2\rightarrow\sigma_1,
$$
$$
\sigma_3\rightarrow v_3\rightarrow\sigma_2,
\qquad
\sigma_4\rightarrow v_4\rightarrow\sigma_3.
$$

The reduced facet matrix $C(g,\chi)$ has six columns
corresponding to the strongly connected components
and four rows indexed by the vertices
$\tau_1,\tau_2,\tau_3,\tau_4$, where distinct
strongly connected components meet. Explicitly,
$$
C(g,\chi)=
\begin{array}{c|rrrrrr}
 & 1 & 2 & 3 & 4 & 5 & 6\\
\hline
\tau_1 & 0 & -1 & 0 & 0 & -1 & 1\\
\tau_2 & -1 & 1 & 1 & 0 & 0 & 0\\
\tau_3 & 0 & 0 & -1 & 1 & 0 & -1\\
\tau_4 & 1 & 0 & 0 & -1 & 1 & 0
\end{array}.
$$
\end{ex}

\begin{ex}
    Below is the matrix $C(f|_{E(p)},\chi)$ for the simplicial distribution $p$ in Example \ref{ex:conncompvertex}.
    $$
    \begin{bmatrix}
        0 & -1 & 0 & 1 & 0 \\
        -1 & 1 & 0 & 0 & 1 \\
        -1 & 1 & 0 & 0 & 1 \\
        -1 & 1 & 0 & 1 & 0 \\
        -1 & 2 & 0 & 0 & 0 \\
        -1 & 2 & 0 & 0 & 0 \\
        0 & 1 & -1 & 0 & 0 \\
        -1 & 2 & 0 & 0 & 0 \\
        1 & -2 & 0 & 0 & 0 \\
        -1 & 1 & 1 & 0 & 0 \\
        1 & -1 & -1 & 0 & 0 \\
        1 & -1 & 0 & -1 & 0 \\
        -1 & 1 & 0 & 0 & 1
    \end{bmatrix}$$
\end{ex}

\begin{pro}\label{pro:rankA=rankC}
Let $f\colon E\to X$ be an $n$-facet-paired bundle scenario. Then, for any direction
$\chi$ on $X$, there is a linear isomorphism
$$
\Null(C(f,\chi))\cong\Null(A(f,\chi)).
$$
\end{pro}

\begin{proof}
Let $E^{(1)},\dots,E^{(k)}$ be the strongly connected
components of $f$, and let $w_1,\dots,w_N$ be the
nondegenerate $(n-1)$-simplices that are faces of
$n$-simplices belonging to distinct strongly connected
components.
By the definition of $C(f,\chi)$, for any $x\in\RR^k$,
we have
\begin{equation}\label{eq:ACmatrixrelation}
(C(f,\chi)x)_i
=
\sum_{j=1}^k
\sum_{e\in E^{(j)}}
A(f,\chi)_{w_i,e}x_j.
\end{equation}

Define a linear map
$$
S\colon\RR^k\longrightarrow\RR^{|E_n^\circ|}
$$
by setting $S(x)_e=x_j$ when $e\in E^{(j)}$.

We first show that $S$ maps $\Null(C(f,\chi))$
into $\Null(A(f,\chi))$.
Let $x\in\Null(C(f,\chi))$ and $w\in E_{n-1}^\circ$.

If $w=w_i$ for some $1\leq i\leq N$,  then using Equation (\ref{eq:ACmatrixrelation}), we obtain
$$
\begin{aligned}
\left( A(f,\chi) S(x) \right)_{w}&=\sum_{e\in E_n^{\circ}}A(f,\chi)_{w_i,e}S(x)_e \\ 
&= \sum_{j=1}^k \sum_{e\in E^{(j)}} A(f,\chi)_{w_i,e} x_j \\
&=\left(C(f,\chi) x\right)_i =0
\end{aligned}
$$
Otherwise, there exist $e_1,e_2 \in E_n^\circ$ such that $e_1\sim_w e_2$,
meaning that $e_1$ and $e_2$ belong to the same strongly connected component. Consequently, we obtain
$$
(A(f,\chi)S(x))_w
=\pm(S(x)_{e_1}-S(x)_{e_2})=0.
$$
Thus $S(x)\in\Null(A(f,\chi))$.
Moreover, $S$ is clearly injective.

To prove surjectivity, let $y\in\Null(A(f,\chi))$.
For $e_1\sim_w e_2$, we have
$$
0=(A(f,\chi)y)_w
=\pm(y_{e_1}-y_{e_2}),
$$
so $y_{e_1}=y_{e_2}$.
Consequently, $y$ is constant on each strongly
connected component.

We may therefore define $x\in\RR^k$ by
$$
x_j=y_e
\qquad\text{for }e\in E^{(j)}.
$$
Using Equation~(\ref{eq:ACmatrixrelation}), we obtain
$$
\begin{aligned}
(C(f,\chi)x)_i
&=\sum_{j=1}^k\sum_{e\in E^{(j)}}
A(f,\chi)_{w_i,e}x_j\\
&=\sum_{e\in E_n^\circ}
A(f,\chi)_{w_i,e}y_e\\
&=(A(f,\chi)y)_{w_i}=0.
\end{aligned}
$$
Hence $x\in\Null(C(f,\chi))$ and $S(x)=y$.
Therefore, $S$ induces the required linear isomorphism.
\end{proof}

\begin{thm}\label{thm:carrCmatrix}
Let $f\colon E\to X$ be a simplicial map, where $X$ is $n$-facet-paired, and let $p\in\sDist(f)$. Denote $g=f|_{E(p)}$. Then, for any choice of direction $\chi$ for $X$,
$$
\dim(\carr(p))=\dim \left(\Null (C(g,\chi)) \right) - 1 .
$$
\end{thm}
\begin{proof}
The result follows directly from
Propositions~\ref{pro:carr=facetmat}
and~\ref{pro:rankA=rankC}.
\end{proof}

Since $p$ is a vertex if and only if $\dim(\carr(p))=0$, we obtain the following corollary:
\begin{cor}
Under the hypotheses of Theorem~\ref{thm:carrCmatrix}, $p$ is a vertex if and only if
$$
\rank \left(C(g,\chi)\right)=k-1,
$$
where $k$ is the number of strongly connected components
of $f|_{E(p)}$.
\end{cor}

\bibliography{bib.bib}
\bibliographystyle{ieeetr}

\normalfont

\end{document}